\documentclass{article}
\usepackage{amsmath, amssymb, amsthm, epsf, graphicx, amscd, amsfonts, amscd, multirow}

\usepackage[pagewise, displaymath, mathlines]{lineno}

\usepackage{setspace}
\newtheorem{theorem}{Theorem}[section]
\newtheorem{lemma}[theorem]{Lemma}

\newtheorem{corollary}[theorem]{Corollary}

\def\bb #1{ {\mathbb #1} }
\def\c #1{ {\mathcal #1} }

\newcommand{\ui}{ {\underline{i} }}

\begin{document}
\title{Meromorphy of the rank one unit root $L$-function revisited}

\author{C. Douglas Haessig}

\maketitle

\abstract{We demonstrate that Wan's alternate description of Dwork's unit root $L$-function in the rank one case may be modified to give a proof of meromorphy that is classical, eliminating the need to study sequences of uniform meromorphic functions.}


\section{Introduction}

To define the unit root $L$-function, we recall some notation from \cite{Wan-DworkConjectureunit-1999}. Let $\bb F_q$ be the finite field with $q = p^a$ elements. Let $K$ be a finite extension field of $\bb Q_p$ with uniformizer $\pi$, ring of integers $R$, and residue field $\bb F_q$. For $u = (u_1, \ldots, u_n) \in \bb Z^n$, define $|u| := |u_1| + \cdots + |u_n|$. Let 
\[
A_0 := \{ \sum_{u \in \bb Z_{\geq 0}^n} a_u x^u \mid a_u \in R, a_u \rightarrow 0 \text{ as } |u| \rightarrow \infty \},
\]
a Banach algebra with sup-norm $|\sum a_u x^u| := \sup | a_u|$. Let $\sigma$ be the $R$-linear map on $A_0$ defined by $\sigma(x^u) = x^{qu}$. Define the Banach $A_0$-module with formal basis $\{ e_i \}_{i = 0}^\infty$:
\[
M := \{ \sum_{i = 0}^\infty a_i e_i \mid a_i \in A_0 \}
\]
with sup-norm $| \sum a_i e_i | := \sup |a_i |$. A \emph{nuclear $\sigma$-module} over $A_0$ is a pair $(M, \phi)$ where $\phi$ is a $\sigma$-linear map, $\phi( \sum a_i e_i ) = \sum \sigma(a_i) \phi(e_i)$, satisfying $\lim_{i \rightarrow \infty} | \phi(e_i) | = 0$.  

For a geometric point $\bar x$ of $\bb G_m^n / \bb F_q$ of degree $d(\bar x) := deg(\bar x) := [\bb F_q(\bar x): \bb F_q]$, denote by $\hat x$  the Teichm\"uller lifting of $\bar x$. The fibre $M_{\hat x}$ at $\bar x$ is a Banach $R[\hat x]$-module with formal basis $\{ e_i \}_{i=0}^\infty$, where $R[\hat x]$ is the ring obtained by adjoining the coordinates of $\hat x$ and
\[
M_{\hat x} := \{ \sum_{i \geq 0} b_i e_i \mid b_i \in R[\hat x] \},
\]
and $\sigma$ acts on $R[\hat x]$ by $\sigma(\hat x) = \hat x^q$. The $d(\bar x)$-th iterate of the fibre map $\phi_{\hat x}^{d(\bar x)}: M_{\hat x} \rightarrow M_{\hat x}$ is an $R[\hat x]$-linear map defined over $R$ whose characteristic function $det(1 - \phi_{\hat x}^{d(\bar x)} T)$ is $p$-adic entire and independent of the geometric point $\bar x$ (but dependent on the closed point over $\bb F_q$ containing $\bar x$). 

Let $B(x)$ be the matrix of $\phi$ with respect to the basis $\{ e_i \}_{i=0}^\infty$, where $B(x)$ acts on column vectors. The nuclear condition on $\phi$ implies that the columns of $B(x)$ become increasingly divisible by $\pi$, and so we may write
\[
B(x) = (C_0, \pi C_1, \pi^2 C_2, \ldots)
\]
where $C_i$ consists of $h_i$ number of columns with entries in $A_0$, and $h_i$ is maximal in the sense that the last column in $C_i$ has an entry which is not divisible by $\pi^{i+1}$.  We say $(M, \phi)$ is \emph{ordinary at slope zero} if at every fiber $\bar x$, $det(1 - \phi_{\hat x}^{d(\bar x)} T)$ has precisely $h_0$ unit roots.

In this paper, we will assume $(M, \phi)$ is ordinary at slope zero, and $h_0 = 1$. In this case, we may factorize 
\begin{equation}\label{E: factor}
det(1 - \phi_{\hat x}^{d(\bar x)} T) = \prod_{i = 0}^\infty (1 - \pi_i(\bar x) T)
\end{equation}
and order the roots such that $ord_p \> \pi_0(\bar x) = 0$ and $ord_p \> \pi_i(\bar x) > 0$ for $i \geq 1$. As explained in Section \ref{S: remarks}, classical exponential sums defined over the torus give rise to such $\sigma$-modules, with the additional property that $\phi e_0 \equiv e_0 \text{ mod}(\pi)$; that is, the unique unit root $\pi_0$ is a 1-unit. Thus, it is natural to assume the \emph{normalization condition}:
\begin{equation}\label{E: normalizeB}
\phi e_0 \equiv e_0 \text{ mod}(\pi) \quad \text{and} \quad \phi e_i \equiv 0 \text{ mod}(\pi) \text{ for all } i \geq 1. 
\end{equation}
It is this extra property that makes the following theory work. The general case, when the unit roots on the fibers are not 1-units, may be reduced to a twist of a $\sigma$-module which satisfies (\ref{E: normalizeB}); see \cite[Section 6]{Wan-DworkConjectureunit-1999} for details on this reduction.

Let $\kappa \in \bb Z_p$. Define the $\kappa$-moment unit root $L$-function
\[
L_\text{unit}(\kappa, \phi, T) := \prod_{\bar x \in | \bb G_m^n / \bb F_q|} \frac{1}{1 - \pi_0(\bar x)^\kappa T^{deg(\bar x)}} \in 1 + T R[[T]].
\]
The definition makes sense since $\pi_0(\bar x)$ is a 1-unit by the normalization condition. For $c \in \bb R$, we say $(M, \phi)$ is  \mbox{\emph{$c$ log-convergent}} if, writing $B(x) = \sum_{u \in \bb Z^n} B_u x^u$, we have 
\[
\lim_{|u| \rightarrow \infty} \inf \frac{ ord_\pi(B_u) }{\log_q|u|} \geq c,
\]
where $ord_\pi(B_u)$ is the minimal order of $\pi$ in all entries of the matrix $B_u$. The main result of Wan \cite{Wan-DworkConjectureunit-1999} is the following, adapted to the case when $\pi_0(\bar x)$ is a 1-unit on each fiber:

\begin{theorem}\label{T: main}\cite[Unit Root Theorem, p. 893]{Wan-DworkConjectureunit-1999}
Let $(M, \phi)$ be a $c$ log-convergent, nuclear $\sigma$-module, ordinary at slope zero of rank one ($h_0 = 1$). Assume it satisfies the normalization condition (\ref{E: normalizeB}). Then $L_\text{unit}(\kappa, \phi, T)$ is $p$-adic meromorphic in $|T| < p^c$ and continuous for $\kappa \in \bb Z_p$.
\end{theorem}

Wan's proof breaks down into two steps. First, he shows that $L_\text{unit}(\kappa, \phi, T)$ is a limit of a sequence of auxiliary $L$-functions coming from symmetric and exterior powers of $\phi$. The second step is to show that this sequence consists of uniformly $p$-adic meromorphic functions in a certain sense, and so the limit $L_\text{unit}(\kappa, \phi, T)$ is meromorphic as desired.

In \cite[Section 8]{Wan-DworkConjectureunit-1999} and \cite[Section 7]{Wan-Rankonecase-2000}, Wan gives an outline of an alternate method using ``limiting $\sigma$-modules''. The main result of this paper is to modify this alternate method in a way such that the proof of Theorem \ref{T: main} is classical, meaning it uses only the classical Dwork trace formula and elementary operator theory. The use of uniform sequences of meromorphic functions is avoided, simplifying the proof.

While the proof uses only classical techniques, we emphasize a distinguishing feature that sets it apart from the classical theory of $L$-functions of exponential sums over finite fields. To describe this, we first need to recall the $p$-adic structure of the latter $L$-function. Let $f$ be a Laurent polynomial in $\bb F_q[x_1^\pm, \ldots, x_n^\pm]$. Associated to the sequence of exponential sums
\[
S_m(f) := \sum_{\bar x \in (\bb F_{q^m}^*)^n} \zeta_p^{Tr_{\bb F_{q^m} / \bb F_p}( f(\bar x))},
\]
where $\zeta_p$ is a primitive $p$-th root of unity, is the $L$-function
\[
L(f, \bb G^n / \bb F_q, T) := \exp\left( \sum_{m=1}^\infty S_m(f) \frac{T^m}{m} \right).
\]
Dwork's theory, as shown by Bombieri \cite{Bombieri-exponentialsumsin-1966}, defines an operator $\alpha$ on a $p$-adic Banach space whose Fredholm determinant satisfies
\[
L(f, \bb G^n / \bb F_q, T)^{(-1)^{n+1}} = det(1 - \alpha T)^{\delta^n}.
\]
Here, $\delta$ is defined by $g(T)^\delta := g(T) / g(qT)$. Since the Fredholm determinant is $p$-adic entire, the $L$-function is $p$-adic meromorphic. Returning to the topic at hand, using (\ref{E: factor}), define
\[
L^{(0)}(\kappa, \phi, T) := \prod_{\bar x \in | \bb G_m^n / \bb F_q|} \> \> \> \prod \frac{1}{1 - \pi_0(\bar x)^{\kappa - r} \pi_{i_1}(\bar x) \cdots \pi_{i_r}(\bar x) T^{deg(\bar x)}},
\]
where the second product runs over all $r \geq 0$ and $1 \leq i_1 \leq i_2 \leq \cdots$. Intuitively, we think of this $L$-function as the $\kappa$-symmetric power of $\phi$. The central point of Section \ref{S: normalized}, and the proof of Theorem \ref{T: main}, is that this $L$-function has the same $\delta$-structure:
\begin{equation}\label{E: delta sym}
L^{(0)}(\kappa, \phi, T)^{(-1)^{n+1}} = det(1 - F_{B^{[\kappa]}} T)^{\delta^n},
\end{equation}
making it also $p$-adic meromorphic (in the appropriate disc). Its relation to the unit root $L$-function is given in Lemma \ref{L: unit root relation}, and as a consequence of this lemma, $L_\text{unit}(\kappa, \phi, T) / L^{(0)}(\kappa, \phi, T)$ is a nowhere zero $p$-adic analytic function on $|T|_p < 1+ \epsilon$ for some $\epsilon > 0$. In particular, by (\ref{E: delta sym}), the zeros and poles of the unit root $L$-function have the same type of $\delta$-structure as the classical $L$-function of exponential sums on a disk $|T|_p < 1+ \epsilon$, but no further due to multiplicities $(-1)^{s-1} (s-1)$ (from Lemma \ref{L: unit root relation}). This may explain why Dwork and Sperber  \cite{Dwork-Sperber} were only able to analytically continue the unit root $L$-function to a disk slightly larger than the unit disk, and why Emerton and Kisin \cite{Emerton-Kisin} were able to describe the zeros and poles of the unit root $L$-function on the closed unit disk.

Some additional remarks are given in Section \ref{S: remarks}.

\section{Meromorphy}\label{S: normalized}

We continue the notation from the introduction. Let $(M, \phi)$ be a nuclear, $c$ log-convergent $\sigma$-module, ordinary at slope zero of rank one ($h_0 = 1$). We will assume the normalization condition (\ref{E: normalizeB}). It follows that the fiber map $\phi_{\hat x}^{d(\bar x)}$ also satisfies $\phi_{\hat x}^{d(\bar x)} e_0 \equiv e_0 \text{ mod}(\pi)$ and $\phi_{\hat x}^{d(\bar x)} e_i \equiv 0 \text{ mod}(\pi)$ for all $i \geq 1$, and thus the eigenvalues
\begin{equation}\label{E: roots}
det(1 - \phi_{\hat x}^{d(\bar x)}  T) = \prod_{i=0}^\infty (1 - \pi_i(\bar x) T), \qquad \pi_i(\bar x) \rightarrow 0 \text{ as } i \rightarrow \infty,
\end{equation}
satisfy: $\pi_0(\bar x) \equiv 1$ mod($\pi$), and $\pi_i(\bar x) \equiv 0$ mod($\pi$) for all $i \geq 1$. 

To $p$-adic analytically continue $L_\text{unit}(\kappa, \phi, T)$ we relate it to the following $L$-functions. For each $s \geq 0$, define
\[
L^{(s)}(\kappa, \phi, T) := \prod_{\bar x \in | \bb G_m^n / \bb F_q|} \> \> \> \prod \frac{1}{1 - \pi_0(\bar x)^{\kappa - r - s} \pi_{i_1}(\bar x) \cdots \pi_{i_r}(\bar x) \cdot \pi_{j_1}(\bar x) \cdots \pi_{j_s}(\bar x) T^{deg(\bar x)}},
\]
where the second product runs over all $r \geq 0$, $1 \leq i_1 \leq i_2 \leq \cdots$ and $0 \leq j_1 < j_2 < \cdots < j_s$. Observe that this product makes sense since there are only finitely many $\bar x$ of bounded degree, and $| \pi_{i_1}(\bar x) \cdots \pi_{i_r}(\bar x) \cdot \pi_{j_1}(\bar x) \cdots \pi_{j_s}(\bar x)| \rightarrow 0$ as $i_1 + \cdots + i_r + j_1 + \cdots + j_s \rightarrow \infty$. Intuitively, $L^{(s)}(\kappa, \phi, T)$ is the $L$-function of $Sym^{\kappa - s} \phi \otimes \wedge^s \phi$.

\begin{lemma}\label{L: unit root relation}
\begin{equation}\label{E: unit root relation}
L_\text{unit}(\kappa, \phi, T) = \prod_{s=0}^\infty L^{(s)}(\kappa, \phi, T)^{(-1)^{s-1} (s-1)}.
\end{equation}
\end{lemma}

\begin{proof}
For convenience, we ignore the $\bar x$ and write $\pi_i$ for $\pi_i(\bar x)$. Also, we will write $\tilde \pi_i := \pi_0^{-1} \pi_i$. Thus, (\ref{E: unit root relation}) is a product of terms with coefficients of the form $\tilde \pi_{i_1} \cdots \tilde \pi_{i_r} \cdot \tilde \pi_{j_1} \cdots \tilde \pi_{j_s} =  \tilde \pi_{i_1}^{e_1} \cdots \tilde \pi_{i_m}^{e_m}$, where we have ignored the factor $\pi_0^\kappa$. In this form, either the product $\tilde \pi_{i_1}^{e_1} \cdots \tilde \pi_{i_m}^{e_m}$ equals 1 (which occurs when each $i_j = 0$), else $i_1 \geq 1$. It is the latter case which we wish to show vanishes from the right-hand side of (\ref{E: unit root relation}). Thus, assume $i_1 \geq 1$, and thus $m \geq 1$.

Observe that $ \tilde \pi_{i_1}^{e_1} \cdots \tilde \pi_{i_m}^{e_m}$ does not appear in $L^{(s)}(\kappa, \phi, T)$ unless $m \geq s-1$. Suppose $m > s-1$. In this case, the term $\tilde \pi_{i_1}^{e_1} \cdots \tilde \pi_{i_m}^{e_m}$ can appear in $L^{(s)}$ by either having $j_1 = 0$, or $j_1 > 0$. If $j_1 = 0$, then the term appears $\binom{m}{s-1}$ times, whereas if $j_1 > 0$, then the term appears $\binom{m}{s}$ times. This means the term appears in $L^{(s)}$ a total number of $\binom{m}{s-1} + \binom{m}{s} = \binom{m+1}{s}$ times. Next, if $m = s-1$, then $\tilde \pi_{i_1}^{e_1} \cdots \tilde \pi_{i_m}^{e_m}$ appears in $L^{(s)}$ a total of $\binom{m}{s-1} = \binom{m}{m}$ times. Thus, taking the exponents into account in (\ref{L: unit root relation}), we see that $\tilde \pi_{i_1}^{e_1} \cdots \tilde \pi_{i_m}^{e_m}$ appears on the right hand side of (\ref{E: unit root relation}) a total number of
\[
\sum_{s=0}^m (-1)^{s-1} (s-1) \binom{m+1}{s} + (-1)^m m \binom{m}{m} = \sum_{s=0}^{m+1} (-1)^{s-1} (s-1) \binom{m+1}{s} = 0,
\]
where the last equality is a binomial coefficient identity. This shows that $ \tilde \pi_{i_1}^{e_1} \cdots \tilde \pi_{i_m}^{e_m}$ vanishes from the right-hand side of (\ref{E: unit root relation}), which leaves only terms coming from $s=0$ and $m = 0$. This proves (\ref{E: unit root relation}).

A more conceptual proof using eigenvalues is given in \cite[Lemma 4.8]{Wan-DworkConjectureunit-1999}.
\end{proof}

The normalization condition (\ref{E: normalizeB}) implies that $L^{(s)}(\kappa, \phi, T) \equiv 1$ mod($\pi^{s-1}$), since, in the definition of $L^{(s)}$, the $\pi_{j_1}, \ldots, \pi_{j_s}$ are distinct. Thus, if we demonstrate that $L^{(s)}(\kappa, \phi, T)$ is $p$-adic meromorphic for $|T| < p^c$, then by (\ref{E: unit root relation}), $L_\text{unit}(\kappa, \phi, T)$ is $p$-adic meromorphic in the same region. We will first focus on the case $s = 0$.

Let $\c S_{\hat x} := R[\hat x][[ e_i : i \geq 1]]$ be the formal power series ring over $R[\hat x]$ with formal basis $\c B := \{1\} \cup \{ e_\ui := e_{i_1} \cdots e_{i_r} \mid r \geq 1, 1 \leq i_1 \leq \cdots \leq i_r \}$, equipped with the sup-norm. (Note the exclusion of $e_0$. We think of 1 as $e_0$, as we will see in the definition of $\Upsilon$.) We may view $M_{\hat x}$ as a subspace of $\c S_{\hat x}$ by defining the map $\Upsilon: M_{\hat x} \rightarrow \c S_{\hat x}$ via $\Upsilon( e_i ) := e_i$ if $i \geq 1$, else $\Upsilon(e_0) := 1$.  By (\ref{E: normalizeB}), we may write $\Upsilon( \phi_{\hat x} e_0 ) = 1 + \eta(\hat x)$ for some $\eta(\hat x) \in \Upsilon(M_{\hat x}) \subset \c S_{\hat x}$ satisfying $|\eta(\hat x)| < 1$. Consequently, $(\Upsilon \circ \phi_{\hat x} e_0)^\kappa = \sum_{i=0}^\infty \binom{\kappa}{i} \eta(\hat x)^i \in \c S_{\hat x}$. Define the map $[ \phi_{\hat x} ]_\kappa: \c S_{\hat x} \rightarrow \c S_{\hat x}$ by
\[
[ \phi_{\hat x} ]_\kappa( e_{i_1} \cdots e_{i_r} ) := ( \Upsilon \circ \phi_{\hat x} e_0)^{\kappa-r} (\Upsilon \circ \phi_{\hat x} e_{i_1}) \cdots (\Upsilon \circ \phi_{\hat x}e_{i_r} ).
\]

Define the length of $e_\ui := e_{i_1} \cdots e_{i_r}$ by $\text{length}( e_\ui ) := r$. For $\xi \in \c S_{\hat x}$, define $\text{length}(\xi)$ as the supremum of the length of its individual terms. In most cases, the length of $\xi$ will be infinite. Write $\kappa = \sum_{i=0}^\infty a_i p^i$ with $a_i \in \{0, 1, \ldots, p-1\}$, and define $k_m := \sum_{i=0}^m a_i p^i$. For each $m \in \bb Z_{\geq 0}$, define the restriction map $[ \phi_{\hat x}]_{\kappa; m}: \c S_{\hat x} \rightarrow \c S_{\hat x}$ by
\[
[ \phi_{\hat x}]_{\kappa; m}( e_\ui) := 
\begin{cases}
[ \phi_{\hat x}]_{k_m}( e_\ui ) & \text{if length($e_\ui ) \leq k_m$} \\
0 & \text{otherwise.}
\end{cases} 
\]

\begin{lemma}\label{L: operator limit}
As operators on $\c S_{\hat x}$, $\lim_{m \rightarrow \infty} [ \phi_{\hat x}]_{\kappa; m} = [ \phi_{\hat x}]_\kappa$.
\end{lemma}

\begin{proof}
To prove the lemma we will need two estimates. The first is, for each $m \geq 0$,
\begin{equation}\label{E: induc}
\left| 1 - ( \Upsilon \circ \phi_{\hat x} e_0)^{\kappa - k_m}  \right| \leq |\pi p^{m+1} |_p.
\end{equation}
To see this, we proceed by induction. Write $\kappa - k_m = p^{m+1} \alpha$ with $\alpha \in \bb Z_p$, and set $\xi :=  ( \Upsilon \circ \phi_{\hat x} e_0)^\alpha$. Then we wish to show $| 1 - \xi^{p^{m+1}}| \leq |\pi p^{m+1}|$ for every $m \geq 0$. Now, the normalization condition $\Upsilon( \phi_{\hat x} e_0)  \equiv 1$ mod($\pi$) implies $\xi \equiv 1$ mod($\pi$). Thus, the case $m = 0$ follows by writing
\[
1-\xi^p = (1 - \xi) (1 + \xi + \cdots + \xi^{p-1})
\]
and observing that $1 + \xi + \cdots + \xi^{p-1} \equiv p$ mod($\pi$). The general case follows similarly by writing
\[
1 - \xi^{p^{m+1}} = (1 - \xi) \cdot \prod_{i = 0}^m (1 + \xi^{p^i} + ( \xi^{p^i} )^2 + \cdots + ( \xi^{p^i} )^{p-1}).
\]

The second estimate we need is,
\[
| [\phi_{\hat x}]_\kappa e_{i_1} \cdots e_{i_r} | \leq | \pi^{r } |,
\]
which follows quickly since $\phi_{\hat x} e_i \equiv 0$ mod($\pi$) for $i \geq 1$. 

We may now prove the lemma. For $e_{i_1} \cdots e_{i_r} \in \c B$, suppose $r \leq k_m$, then
\begin{align*}
\left( [\phi_{\hat x}]_{\kappa; m} - [\phi_{\hat x}]_\kappa \right) e_{i_1} \cdots e_{i_r} &= \left( ( \Upsilon \circ \phi_{\hat x} e_0)^{k_m - r} - ( \Upsilon \circ \phi_{\hat x} e_0 )^{\kappa-r} \right) ( \Upsilon \circ \phi_{\hat x} e_{i_1}) \cdots ( \Upsilon \circ \phi_{\hat x} e_{i_r})  \\
&= \left( 1 - \left( \Upsilon \circ \phi_{\hat x}e_0 \right)^{\kappa - k_m}  \right) \cdot [\phi_{\hat x}]_{\kappa; m} e_{i_1} \cdots e_{i_r}.
\end{align*}
Thus, writing $e_\ui := e_{i_1} \cdots e_{i_r}$,
\[
\left| \left( [\phi_{\hat x}]_{\kappa; m} - [\phi_{\hat x}]_\kappa \right) e_\ui \right| \leq \left|  1 - ( \Upsilon \circ \phi_{\hat x} e_0 )^{\kappa - k_m}  \right| \cdot \left|  [\phi_{\hat x}]_{\kappa; m} e_\ui \right|,
\]
which tends to zero as $m \rightarrow \infty$ by (\ref{E: induc}).

Next, if length($e_\ui ) = r  > k_m$ then by definition, $[\phi_{\hat x}]_{\kappa; m} e_\ui = 0$. Hence,
\[
\left| \left( [\phi_{\hat x}]_{\kappa; m} - [\phi_{\hat x}]_\kappa \right) e_\ui \right| = \left| [\phi_{\hat x}]_k e_\ui \right| \leq |\pi^r |,
\]
which tends to zero as $m$ grows since $r > k_m$.
\end{proof}

\begin{lemma}\label{L: det equal}
We have the following properties:
\begin{enumerate}
\item As operators on $\c S_{\hat x}$, for each positive integer $j$, $\left( [ \phi_{\hat x} ]_{\kappa; m} \right)^j =  [\phi_{\hat x}^{j} ]_{\kappa; m}$.  Consequently,  $\left( [ \phi_{\hat x}^{d(\bar x)} ]_{\kappa; m} \right)^j =  [\phi_{\hat x}^{j d(\bar x)} ]_{\kappa; m}$. 
\item $det(1 - [\phi_{\hat x}^{d(\bar x)}]_{\kappa; m} T \mid \c S_{\hat x}) = \prod (1 - \pi_0(\bar x)^{k_m - r} \pi_{i_1}(\bar x) \cdots \pi_{i_r}(\bar x) T)$, where the product runs over all $0 \leq r \leq k_m$, $1 \leq i_1 \leq i_2 \leq \cdots$.
\end{enumerate}
\end{lemma}

\begin{proof}
We first make the identification $\{ \xi \in \c S_{\hat x} \mid \text{length}(\xi) \leq k_m \} \cong Sym^{k_m} M_{\hat x}$ by sending
\[
e_{i_1} \cdots e_{i_r} \longmapsto  e_0^{k_m - r} e_{i_1} \cdots e_{i_r}.
\]
Using this identification, we have
\begin{align*}
[\phi_{\hat x}]_{\kappa; m}( e_{i_1} \cdots e_{i_r}) &=  (\Upsilon \circ \phi_{\hat x} e_0)^{k_m - r} ( \Upsilon \circ \phi_{\hat x} e_{i_1}) \cdots ( \Upsilon \circ \phi_{\hat x} e_{i_r}) ) \\
&\cong \left( Sym^{k_m} \phi_{\hat x}  \right)( e_0^{k_m - r} e_{i_1} \cdots e_{i_r}),
\end{align*}
and so $[\phi_{\hat x}]_{\kappa; m} \cong  Sym^{k_m} \phi_{\hat x}$. Thus, for any $j$ a positive integer,
\[
\left( [\phi_{\hat x}]_{\kappa; m} \right)^j \cong  \left( Sym^{k_m} \phi_{\hat x}  \right)^j =  Sym^{k_m}\phi_{\hat x}^{j}  \cong [\phi_{\hat x}^j ]_{\kappa; m},
\]
which proves the first part of the lemma. Setting $j = d(\bar x)$ proves the second part of the lemma using (\ref{E: factor}).
\end{proof}

\begin{corollary}\label{C: properties}
We have the following properties:
\begin{enumerate}
\item As operators on $\c S_{\hat x}$, for each positive integer $j$, $\left( [ \phi_{\hat x} ]_{\kappa} \right)^j =  [\phi_{\hat x}^{j} ]_{\kappa}$.  Consequently,  $\left( [ \phi_{\hat x}^{d(\bar x)} ]_{\kappa} \right)^j =  [\phi_{\hat x}^{j d(\bar x)} ]_{\kappa}$. 
\item $det(1 - [\phi_{\hat x}^{d(\bar x)}]_{\kappa} T \mid \c S_{\hat x}) = \prod (1 - \pi_0(\bar x)^{\kappa - r} \pi_{i_1}(\bar x) \cdots \pi_{i_r}(\bar x) T)$, where the product runs over all $r \geq 0$, $1 \leq i_1 \leq i_2 \leq \cdots$.
\end{enumerate}
\end{corollary}

\begin{proof}
From Lemma \ref{L: det equal}, $[\phi_{\hat x}^{j}]_{\kappa; m} = \left( [\phi_{\hat x}]_{\kappa; m} \right)^j$ for every $m \geq 0$. From Lemma \ref{L: operator limit}, we have
\[
[\phi_{\hat x}^{j}]_\kappa = \lim_{m \rightarrow \infty} [\phi_{\hat x}^{j}]_{\kappa; m} = \lim_{m \rightarrow \infty} \left( [\phi_{\hat x}]_{\kappa; m} \right)^j = \left( [\phi_{\hat x}]_\kappa \right)^j,
\]
since $\phi^j$ also satisfies the normalization condition. This proves the first part of the lemma.

The second part follows by
\begin{align*}
\prod_{r \geq 0, 1 \leq i_1 \leq i_2 \leq \cdots} (1 - \pi_0(\bar x)^{\kappa - r} \pi_{i_1}(\bar x) \cdots \pi_{i_r}(\bar x) T) &= \lim_{m \rightarrow \infty} \prod_{\substack{1 \leq i_1 \leq i_2 \leq \cdots \\ \text{such that } 0 \leq r \leq k_m}} (1 - \pi_0(\bar x)^{k_m - r} \pi_{i_1}(\bar x) \cdots \pi_{i_r}(\bar x) T) \\
&= \lim_{m \rightarrow \infty} det(1 - [ \phi_{\hat x}^{d(\bar x)} ]_{\kappa; m} T \mid \c S_{\hat x} ) \\
&=  det(1 - [\phi_{\hat x}^{d(\bar x)}]_\kappa T \mid \c S_{\hat x}) \quad \text{by Lemma \ref{L: operator limit}},
\end{align*}
where the above limits mean convergence of coefficients.
\end{proof}

Consequently,
\begin{equation}\label{E: L0}
L^{(0)}(\kappa, \phi, T) = \prod_{\bar x \in |\bb G_m^n / \bb F_q|} \frac{1}{det(1 - [\phi_{\hat x}^{d(\bar x)}]_\kappa T^{deg(\bar x)} \mid \c S_{\hat x})}.
\end{equation}

We now switch to working with the family rather than its fibers. Denote by $\c S := A_0[[ e_i : i \geq 1]]$ the formal power series ring over $A_0$ with formal basis $\c B$, equipped with the sup-norm. We may view $M$ as a subspace of $\c S$ by the following map. Abusing notation, define $\Upsilon: M \rightarrow \c S$ by $\Upsilon( e_i ) := e_i$ if $i \geq 1$ and $\Upsilon(e_0) := 1$.  By (\ref{E: normalizeB}), we may write $\Upsilon(\phi e_0) = 1 + \eta$ for some $\eta \in \Upsilon(M) \subset \c S$ satisfying $|\eta| < 1$. Consequently, $(\Upsilon \circ \phi e_0)^\kappa = \sum_{i=0}^\infty \binom{\kappa}{i} \eta^i \in \c S$.  Define the map $[ \phi ]_\kappa: \c S \rightarrow \c S$ by
\[
[ \phi ]_\kappa( e_{i_1} \cdots e_{i_r}) := ( \Upsilon \circ \phi e_0)^{\kappa-r} (\Upsilon \circ \phi e_{i_1}) \cdots (\Upsilon \circ \phi e_{i_r} ).
\]
Notice that $(\c S, [\phi]_\kappa)$ is a nuclear $c$ log-convergent $\sigma$-module. 

To make use of the Dwork trace formula, we need to momentarily switch to matrices. Let $B^{[\kappa]}$ denote the matrix of $[\phi]_\kappa$ with respect to the basis $\c B$ (acting on column vectors), and observe that
\[
\text{matrix of } ( [\phi]_\kappa )^m =  B^{[\kappa]}(x^{q^{m-1}}) \cdots B^{[\kappa]}(x^q) B^{[\kappa]}(x) .
\]
Write
\[
B^{[\kappa]}(x) = \sum_{u \in \bb Z^n} B_u^{[\kappa]} x^u,
\]
where $B_u^{[\kappa]}$ is a matrix with entries in $R$. Define the block matrix $F_{B^{[\kappa]}} := ( B_{qu - v}^{[\kappa]} )_{u, v \in \bb Z^n}$. Now, for every $m \geq 1$,
\begin{align}\label{E: Dwork trace}
(q^m - 1)^n Tr( F_{B^{[\kappa]}}^m ) &= \sum_{\substack{ \bar x \in (\bb F_{q^m}^*)^n \\ \hat x = \text{Teich}(\bar x)}}  Tr( B^{[\kappa]}(\hat x^{q^{m-1}}) \cdots B^{[\kappa]}(\hat x^q) B^{[\kappa]}(\hat x) )  \quad \text{by the Dwork trace formula \cite[Lemma 4.1]{Wan-p-adic-representation}} \notag \\
&= \sum_{\substack{ \bar x \in (\bb F_{q^m}^*)^n \\ \hat x = \text{Teich}(\bar x)}}  Tr \left( [\phi_{\hat x}]_\kappa \right)^m.
\end{align}
For $g(T)$ any function, define $\delta$ by $g(T)^\delta := g(T) / g(qT)$. Then we calculate:
\begin{align*}
det(1 - F_{B^{[\kappa]}} T)^{\delta^n \cdot (-1)^{n +1}} &= \exp\left( \sum_{m=1}^\infty (q^m - 1)^n Tr( F_{B^{[\kappa]}}^m ) \frac{T^m}{m} \right) \\ 
&= \exp\left( \sum_{m=1}^\infty \sum_{\hat x^{q^m - 1} = 1} Tr( [ \phi_{\hat x}]_\kappa )^m \frac{T^m}{m} \right)  \quad \text{ by (\ref{E: Dwork trace})} \\
&= \exp\left( \sum_{r=1}^\infty \> \sum_{\substack{\bar x \in (\overline{\bb F}_q^*)^n \\ deg(\bar x) = r}} \> \sum_{s=1}^\infty Tr( [ \phi_{\hat x}^r ]_\kappa )^s \frac{T^{rs}}{rs} \right) \quad \text{by Corollary \ref{C: properties}, part 1} \\
&= \prod_{\bar x \in ( \overline{\bb F}_q^*)^n} \exp\left( \sum_{s=1}^\infty Tr \left( [ \phi_{\hat x}^{d(\bar x)} ]_\kappa \right)^{s} \frac{T^{s d(\bar x)}}{ s \cdot d(\bar x)} \right) \\
&=  \prod_{\bar x \in ( \overline{\bb F}_q^*)^n} \left( \frac{1}{det(1 - [ \phi_{\hat x}^{d(\bar x)} ]_\kappa T^{deg(\bar t)} )} \right)^{1 / d(\bar x)} \\
&= \prod_{\bar x \in |\bb G_m^n / \bb F_q|} \frac{1}{det(1 - [ \phi_{\hat x}^{d(\bar x)} ]_\kappa T^{deg(\bar x)})} \\
&= L^{(0)}(\kappa, \phi, T) \quad \text{by (\ref{E: L0})}.
\end{align*}
Since $(\c S, [\phi]_\kappa)$ is $c$ log-convergent, by \cite[Proposition 3.6]{Wan-p-adic-representation}, $det(1 - F_{B^{[\kappa]}} T)$ converges for $|T| < p^c$, and hence, $L^{(0)}(\kappa, \phi, T)$ is $p$-adic meromorphic in $|T| < p^c$. 

A completely analogous argument shows, for $s \geq 2$,
\[
L^{(s)}(\kappa, \phi, T) = \prod_{\bar x \in |\bb G_m^n / \bb F_q|} \frac{1}{det \left(1 - [\phi_{\hat x}^{d(\bar x)}]_{\kappa-s} \otimes \wedge^s \phi_{\hat x}^{d(\bar x)} T^{deg(\bar x)} \mid \c S_{\hat x} \otimes \wedge^s M_{\hat x}  \right)},
\]
and
\[
L^{(s)}(\kappa, \phi, T)^{(-1)^{n+1}} = det(1 - F_{B^{[\kappa-s]} \otimes \wedge^s B} T)^{\delta^n},
\]
where $B^{[\kappa-s]} \otimes \wedge^s B$ is the matrix of $[\phi]_{\kappa-s} \otimes \wedge^s \phi$. Since $[\phi]_{\kappa-s} \otimes \wedge^s \phi$ is $c$ log-convergent, $L^{(s)}(\kappa, \phi, T)$ is $p$-adic meromorphic on $|T| < p^c$. This proves the meromorphy portion of Theorem \ref{T: main} by the paragraph after Lemma \ref{L: unit root relation}.

Let us now show continuity in $\kappa$. Observe that we may write $[\phi]_{\kappa} = (\Upsilon \circ \phi e_0)^\kappa \circ [\phi]_0$, where the map $(\Upsilon \circ \phi e_0)^\kappa$ acts on $\c S$ via multiplication. Writing $\Upsilon \circ \phi e_0 = 1 + \eta$ for some $\eta \in \Upsilon(M) \subset \c S$ with $|\eta| < 1$, we see that $[\phi]_\kappa$ is continuous for $\kappa \in \bb Z_p$, and thus $L^{(0)}(\kappa, \phi, T)$ is continuous for $\kappa \in \bb Z_p$. Similarly, 
\[
[\phi]_{\kappa-s} \otimes \wedge^s \phi =  \left( (\Upsilon \circ \phi e_0)^\kappa \circ [\phi]_{-s} \right) \otimes \wedge^s \phi.
\]
It follows that $L_\text{unit}(\kappa, \phi, T)$ is continuous in $\kappa \in \bb Z_p$.  We note that Gross-Klo\"nne \cite{Grosse-Klonne} has shown that $\kappa$ may be extended outside the unit disk via characters. This completes the proof of Theorem \ref{T: main}.

\section{Remarks}\label{S: remarks}

\noindent{\bf Remark 1.}  One of the most common, non-trivial unit root $L$-functions is the classical $L$-function associated to exponential sums defined over a torus. That is, associated to the Laurent polynomial $f(x) = \sum_{u \in \bb Z^n} a_u x^u \in \bb F_q[x_1^\pm, \ldots, x_n^\pm]$ is its $L$-function
\begin{align}\label{E: old unit root}
L(f, \bb G^n / \bb F_q, T) &:= \prod_{\bar x \in | \bb G_m^n / \bb F_q|} \frac{1}{1 - \zeta_p^{Tr_{\bb F_{q^{d(\bar x)}} / \bb F_p}( f(\bar x) )} T^{deg(\bar x)}} \\
&= \exp\left( \sum_{m=1}^\infty S_m(f) \frac{T^m}{m} \right), \notag
\end{align}
where $\zeta_p$ is a primitive $p$-th root of unity and
\[
S_m(f) := \sum_{\bar x \in (\bb F_{q^m}^*)^n} \zeta_p^{Tr_{\bb F_{q^m} / \bb F_p}( f(\bar x))}.
\]
The $\sigma$-module attached to this is defined as follows. 

Let $E(z) := \exp\left( \sum_{i=0}^\infty z^{p^i} / p^i \right)$ be the Artin-Hasse exponential. Fix a root $\pi$ of $\sum_{i=0}^\infty z^{p^i} / p^i$ such that $ord_p(\pi) = -1 / (p-1)$, and define Dwork's splitting function by $\theta(z) := E(\pi z)$. Three properties are satisfied: $\theta(z)$ converges for $|z|_p < 1 + \epsilon$,  $\theta(1)$ is a primitive $p$-th root of unity, and the splitting property, for $\bar t \in \bb F_{p^a}$ and $\hat t \in \overline{\bb Q}_p$ its Teichm\"uller lift, then $\theta(1)^{Tr_{\mathbb{F}_{p^a}/\mathbb{F}_p}(\bar{t})}=\theta(\hat{t})\theta(\hat{t}^p)\dots\theta(\hat{t}^{p^{a-1}})$. The last property implies $\bar t \mapsto \theta(\hat{t})\theta(\hat{t}^p)\cdots\theta(\hat{t}^{p^{a-1}})$ is a $p$-adic analytic lifting of an additive character on $\bb F_{p^a}$.

With $q = p^a$, denote by $\bb Q_q$ the unique unramified extension of degree $a$ over $\bb Q_p$, and let $\bb Z_q$ be its ring of integers. Using notation from the introduction, set $R = \bb Z_q[\pi]$ and $M = A_0 e_0$ where $e_0$ is a formal basis element, and define $\phi: M \rightarrow M$ by $e_0 \mapsto F_a(x) e_0$, where
\[
F(x) := \prod_{u \in Supp(f)} \theta(\hat a_u x^u) \quad \text{and} \quad F_a(x) := \prod_{i=0}^{a-1} F^{\tau^i}( x^{p^i}),
\]
and $\tau \in Gal(\bb Q_q / \bb Q_p)$ is the lifting of the Frobenius generator of $Gal(\bb F_q / \bb F_p)$, and $\hat a_u$ is the Teichm\"uller lifting of $a_u$. Here, $\sigma$ is $\tau$-linear on $A_0$  defined by $\sigma(c_u x^u) = \tau(c_u) x^{qu}$.  Observe that $\phi$ satisfies the normalization condition (\ref{E: normalizeB}) with uniformizer $\pi$. Then
\begin{align*}
det(1 - \phi_{\hat x}^{d(\bar x)} T) &= 1 - F_a(\hat x) F_a(\hat x^q) \cdots F_a(\hat x^{q^{d(\bar x) - 1}}) T \\
&= 1 - \zeta_p^{Tr_{\bb F_{q^{d(\bar x)}} / \bb F_p}( f(\bar x))} T.
\end{align*}
Hence, using notation from the previous section, setting $\kappa = 1$ we have $L^{(0)}(1, \phi, T) = L(f, \bb G_m^n / \bb F_q, T)$. Further, since $M$ is rank one, $L^{(s)}(1, \phi, T) = 1$ for all $s \geq 2$, and thus $L_\text{unit}(1, \phi, T) = L^{(0)}(1, \phi, T) = L(f, \bb G_m^n / \bb F_q, T)$ and $L(f, \bb G_m^n / \bb F_q, T)^{(-1)^{n+1} } = det(1 - F_B^{[1]} T)^{\delta^n}$. Of course, it is well-known that this is a rational function.

\bigskip\noindent{\bf Remark 2.} The previous remark suggests replacing the 1-unit $\zeta_p$ with any other 1-unit to get a similar result. That is, let $\alpha \in \overline{\bb Q}_p$ be a 1-unit. Let $f \in \bb Z_q[x_1^\pm, \ldots, x_n^\pm]$,  and define
\[
L(f, \alpha, T) := \prod_{\bar x \in | \bb G_m^n / \bb F_q|} \frac{1}{1 - \alpha^{Tr_{\bb Q_{q^{d(\bar x)}} / \bb Q_p}( f(\hat x) )} T^{deg(\bar x)}}.
\]
In this case, Dwork shows us how to define a splitting function for this in \cite[Equation (2)]{Dwork-rationality_of_zeta_function} (or \cite[Ch. V, Sec. 2]{Koblitz_p-adic_book}) as follows. Write $\alpha = 1 + \eta$, with $ord_p(\eta) > 0$. Define
\[
F(X, Y) := (1 + Y)^X \cdot \prod_{j=1}^\infty (1 + Y^{p^j})^{(X^{p^j} - X^{p^{j-1}}) / p^j},
\]
and set $\theta(z) := F(z, \eta)$. Then $\theta(z)$ is a splitting function satisfying the three properties: $\theta(z)$ converges for $|z|_p < 1 + \epsilon$,  $\theta(1) = \alpha$, and the splitting property, for $\bar t \in \bb F_{p^a}$ and $\hat t \in \overline{\bb Q}_p$ its Teichm\"uller lift, then $\alpha^{Tr_{\mathbb{Q}_{p^a}/\mathbb{Q}_p}(\hat t)}=\theta(\hat{t})\theta(\hat{t}^p)\dots\theta(\hat{t}^{p^{a-1}})$. Defining $M$ and $\phi$ as in the previous remark, then $L_\text{unit}(1, \phi, T) = L(f, \alpha, T)$. Viewing $\alpha$ generically by replacing $\alpha$ with $1+Z$ where $Z$ is an indeterminant has been investigated in \cite{Wan_T-adic}.

\bibliographystyle{amsplain}
\bibliography{References-1}

\end{document}